\documentclass[12pt,a4paper,reqno]{amsart}
\usepackage[T2A]{fontenc}

\usepackage{hyphenat}

\usepackage[english]{babel}

\usepackage{xcolor}
\usepackage{hyperref}

\usepackage{amsmath, amssymb, amsthm, mathtools, thmtools, mathdots, graphicx}
\usepackage[shortlabels]{enumitem}

\definecolor{linkcolor}{HTML}{970505}
\definecolor{urlcolor}{HTML}{799B03} 
\definecolor{citecolor}{HTML}{799B03}

\hypersetup{pdfstartview=FitH,  linkcolor=linkcolor,urlcolor=urlcolor, citecolor = citecolor,  colorlinks=true }

\DeclareMathOperator{\Hess}{Hess}
\DeclareMathOperator{\Span}{span}

\let\Re\barbaz
\let\Im\quuuz

\DeclareMathOperator{\Re}{Re}
\DeclareMathOperator{\Im}{Im}
\DeclareMathOperator{\grad}{grad}
\DeclareMathOperator{\dive}{div}

\usepackage{graphicx}
\usepackage[all]{xy}
\usepackage[left=3cm,right=3cm,top=3cm,bottom=3cm,bindingoffset=0cm]{geometry}
1

\theoremstyle{definition}

\newtheorem{theoremf}{Theorem}[section] 
\newtheorem{def1}[theoremf]{Definition}

\newtheorem{remark}[theoremf]{Remark}

\newtheorem{propositionf}[theoremf]{Proposition}

\newtheorem{example}[theoremf]{Example}

\newtheorem*{lemmaf*}{Лемма}
\newtheorem*{theoremf*}{Теорема}
\newtheorem*{remark*}{Замечание}
\newtheorem*{statementf*}{Утверждение}
\newtheorem*{propositionf*}{Предложение}
\newtheorem*{problem*}{Задача}
\newtheorem*{example*}{Пример}
\newtheorem*{notation*}{Обозначение}

\usepackage{import}
\usepackage{xifthen}
\usepackage{pdfpages}
\usepackage{transparent}

\begin{document}
	
\title{Minimal submanifolds in spheres and complex-valued eigenfunctions}
\author[1]{Aleksei Kislitsyn}

\address{Department of Higher Geometry and Topology, Faculty
of Mathematics and Mechanics, Moscow State University, Leninskie Gory,
GSP-1, 119991, Moscow, Russia}
\address{Independent University of Moscow, Bolshoy
Vlasyevskiy Pereulok 11, 119002, Moscow, Russia}
\email{aleksei.kislitcyn@math.msu.ru}

\begin{abstract}
A new approach for constructing minimal submanifolds of codimension 1 in the round spheres  is proposed.  In the case of  $\mathbb{S}^3$ two immersions of the Clifford  torus and all Lawson $\tau_{n, m}$ surfaces are  described in terms of $(\lambda, \mu)$-eigenfunctions. Also, a new proof of a theorem that describes $(\lambda, \mu)$-eigenfunctions on sphere is obtained. This proof is based on a statement that a function $f$ is a $(\lambda, \mu)$-eigenfunction if and only if $f$ and $f^2$ are eigenfunctions for the Laplace-Beltrami operator. 
\end{abstract}

\maketitle

\section{Introduction}

Minimal surfaces and minimal submanifolds are very classic objects of study in Differential Geometry, see 
e.g.~\cite{DierkesHildebrandtSauvigny2010, DierkesHildebrandtTromba2010, DierkesHildebrandtTromba2010a, Xin2003}.
During last 18 years a special case of minimal submanifolds in the spheres $\mathbb{S}^n$ with canonical metrics attracted a lot
of interest due to its connection to Spectral Geometry. As it was proved by Nadirashvili~\cite{Nadirashvili1996},
El Soufi and Ilias~\cite{ElSoufiIlias2008}, the critical metrics for eigenvalues of the Laplace-Beltrami operator correspond
exactly to the metrics induced on minimally immersed submanifolds of spheres. A lot of progress in the problem of sharp
isoperimetric inequalities for eigenvalues of the Laplace-Bletrami operator on surfaces was achieved due to this connection,
see e.g.~\cite{Nadirashvili1996, JakobsonNadirashviliPolterovich2006, ElSoufiGiacominiJazar2006, CianciKarpukhinMedvedev2019, JakobsonLevitinNadirashviliNigamPolterovich2005, NayataniShoda2019,KarpukhinNadirashviliPenskoiPolterovich2021,NadirashviliPenskoi2018, Karpukhin2021}. 
More about critical metrics on surfaces and minimal immesrions in $\mathbb{S}^n$ could be found in 
reviews~\cite{Penskoi2013b,Penskoi2019}.
  
Examples of  implicit minimal submanifolds are known from the paper \cite{Cartan1939} by Cartan, where families of minimal and constant mean curvature surfaces were constructed. Baird and Gudmundsson  in the paper \cite{MR1190447}  found a criterion for minimality of $f^{-1}(y)$ where the map $f: (M, g) \rightarrow (N, h)$  is horizontally conformal up to the first order along $f^{-1}(y)$ (see Definition  \hyperlink{def2}{2}). Later in \cite{gudmundsson2023minimal} Gudmundsson and Munn  distinguished a class of such maps, called complex-valued eigenfunctions or $(\lambda, \mu)$-eigenfunctions. 

Let $(M, g)$ be a Riemannian manifold,  let $T^{\mathbb{C}}M$ be  the complexification of the $TM$ and $g$ be the extension of the metric  to a complex bilinear form on $T^{\mathbb{C}}M$. We introduce \textit{Laplace-Beltrami operator} $\tau$ and bilinear \textit{conformality operator} $\varkappa$, that act on complex-valued functions $\varphi, \psi: (M, g) \rightarrow \mathbb{C}$  as follows,
$$\tau(\varphi) := \dive \grad \varphi, \quad \varkappa(\varphi, \psi) := g(\grad \varphi, \grad \psi).$$ 

  \begin{def1}[\cite{gudmundsson2023minimal}]

   A function $f: (M, g) \rightarrow \mathbb{C}$ is called a \textit{complex-valued eigenfunction} or a \textit{$(\lambda, \mu)$-eigenfunction} if there exist $\lambda, \, \mu \in \mathbb{C}$ such that
$$ \tau(f) = \lambda f, \quad \varkappa(f, f) = \mu f^2.$$

    A set $\mathcal{E} = \{ f_i : M \rightarrow \mathbb{C} \, | \, i \in I \}$ of complex-valued functions is said to be an \textit{eigenfamily} on $M$ if there exist complex numbers $\lambda, \, \mu \in \mathbb{C}$ such that for all $f_i, \, f_j \in \mathcal{E}$ we have
    $$ \tau(f_i) = \lambda f_i, \quad \varkappa(f_i, f_j) = \mu f_i f_j.$$
  \end{def1}

  \begin{theoremf}[\cite{gudmundsson2023minimal}]
    \label{GM_th}
 Let $f: (M, g) \rightarrow \mathbb{C}$ be a complex-valued eigenfunction on a Riemannian manifold, such that $0 \in f(M)$ is a regular value for $f$. Then the fiber $f^{-1}(0)$ is a minimal submanifold of $M$ of codimension two. 
      
  \end{theoremf}

    In this paper we apply $(\lambda, \mu)$-eigenfunctions in order to find minimal submaifolds in spheres. Firstly, in Section 2 we propose a proof of \cite[Theorem 1]{riedler2023polynomial} in the particular case of  $(\lambda, \mu)$-eigenfunctions. Our proof is shorter and uses another approach than the original one.

    \begin{theoremf}
    \label{Sn_class}
        Let $n \geq 2$, then a function $f: \mathbb{S}^n \rightarrow \mathbb{C}$ is a $(\lambda, \mu)$-eigenfunction if and only if there exists a homogeneous polynomial $F: \mathbb{R}^{n+1} \rightarrow \mathbb{C}$ such that $f = F|_{\mathbb{S}^n}$ and the polynomial $F$ is a $(0, 0)$-eigenfunction.

    \end{theoremf}

 In Section 3 some examples of $(\lambda, \mu)$-eigenfunctions on $\mathbb{S}^4$ are presented. However, these functions give only trivial examples of minimal submanifolds in sphere, namely flat sections of a sphere. Also it is shown that in fact,  the minimal submanifolds corresponding to polynomials of degree less or equal to 3  are trivial. 

 In Section 4 we suggest a way to construct minimal submanifolds of codimension 1 in spheres with help of $(\lambda, \mu)$-eigenfunctions.  We should remark  that $(0, 0)$-eigenfunctions were implicitly used to construct minimal submanifolds of codimension 1 in $\mathbb{R}^n$ \cite[Example 3.5]{MR1190447}. But to the best of the author knowledge minimal submanifolds of codimension 1 were not described using  $(\lambda, \mu)$-eigenfunctions in other manifolds except $\mathbb{R}^n$.  In this paper we propose a way of constructing such manifolds in $\mathbb{S}^n$.
 
 For every $\alpha \in \mathbb{C}$ such that $|\alpha| = 1$  we denote  by $l_{\alpha}$ the real straight line in $\mathbb{C}$ given by the equation $a \, u + b \, v = 0$, where $a := \Re \alpha, \, \, b := \Im \alpha$.  Then the following Theorem is proved.

\begin{theoremf}
\label{Sn_th}
        Let  $f: \mathbb{S}^n \rightarrow \mathbb{C}$ be a $(\lambda, \mu)$-eigenfunction given by a polynomial $F$. Suppose that $0$ is a regular value of $f$. Define $ P_{\alpha} := a \Re F + b \Im F$, then $f^{-1}(l_\alpha)$ is minimal in $\mathbb{S}^n$ if and only if for all $x \in P^{-1}_{\alpha}(0)$ we have
\begin{equation} \label{cond}
\Hess P_\alpha (x) \Big (\grad P_\alpha (x), \grad P_\alpha (x) \Big) = 0.
\end{equation}

\end{theoremf}
In the case of $\mathbb{S}^3$ we can simplify  condition (\ref{cond}) and describe how the submanifold $f^{-1}(l_\alpha)$ depends on $\alpha$.

\begin{theoremf}
\label{S3_th}

        Let  $f: \mathbb{S}^3 \rightarrow \mathbb{C}$ be a $(\lambda, \mu)$-eigenfunction given by a polynomial $F$. Suppose that $0$ is a regular value of $f$. Then 
        \begin{enumerate}
            \item for all $\alpha, \beta$ there exists an isometry of $\mathbb{R}^4$ that induces  an isometry between $f^{-1}(l_\alpha)$ and $f^{-1}(l_\beta)$;
            \item the submanifold $f^{-1}(l_\alpha)$ is minimal in $\mathbb{S}^3$ for all $\alpha$ if and only if for some $\alpha_0$ condition (\ref{cond}) holds.
       
        \end{enumerate}

\end{theoremf}

The second statement of Theorem \ref{S3_th} shows that it is  enough to check condition~(\ref{cond}) only for $P_1 = \Re F$ or $P_i = \Im F$. Theorem \ref{S3_th} allows to describe immersions of the Clifford  torus and the  Lawson $\tau_{n, m}$ surfaces  using  $(\lambda, \mu)$-eigenfunction. Looking for other examples is an interesting problem for future  investigation.

\section{Complex-valued eigenfunctions on sphere  }

\subsection{Auxiliary proposition}
\label{propos}
In order to prove Theorem \ref{Sn_class} we need to describe $(\lambda, \mu)$-eigenfunctions in terms of eigenfunctions of the Laplace–Beltrami operator. There is a well known identity for $\tau$,
\begin{equation} \label{id}
\tau(\varphi \psi) = \tau(\varphi) \psi + 2 \varkappa(\varphi, \psi) + \varphi \tau(\psi).
\end{equation}

\begin{propositionf}

    Let $(M, g)$ be a Riemannian manifold and let $f$ be a function on it. Then the following conditions are equivalent.

     \begin{enumerate}[1)]
        \item A function $f$ is a $(\lambda, \mu)$-eigenfunction.
        \item Functions $f$ and $f^2$ are eigenfunctions of the operator $\tau$.
        \item Functions $f^n$, $ n \in \mathbb{N} $ are eigenfunctions of the operator $\tau$.
    \end{enumerate}
\end{propositionf}

\begin{proof}
    Let us prove that 2) implies 1). Suppose that functions $f$ and $f^2$ are eigenfunctions with eigenvalues $\lambda_1$ and $\lambda_2$.  Substituting $\psi = f, \,  \varphi = f$ into \eqref{id}, we get

    $$\tau(f^2) =  \tau(f) f + 2 \varkappa(f, f) + f \tau(f). $$
    Then, we have

    $$\lambda_2 f^2 = 2 \lambda_1 f^2 + 2 \varkappa(f, f).$$
    So, a function $f$ is a $(\lambda, \mu)$-eigenfunction with $\lambda = \lambda_1$, $\mu = \frac{\lambda_2}{2} - \lambda_1$.

     Let us prove that 1) implies 3). By applying \cite[Theorem 2.4]{gudmundsson2023minimal} to the eigenfamily $ \mathcal{E} = \{ f \}$ we see that $f^n$ are complex-valued eigenfunctions. In particular, $f^n$ are eigenfunctions of the operator $\tau$.
     
     It is clear that 3) implies 2).

\end{proof}

\subsection{Proof of  Theorem \ref{Sn_class}}

\begin{proof}
    
Let us prove the necessity. Consider a   $(\lambda, \mu)$-eigenfunction $f$. Then by Proposition \ref{propos} functions $f$ and $f^2$ are eigenfunction for $\tau$. It follows from the description of eigenfunctions on sphere that there exist homogeneous harmonic polynomials $F, \, F_2$ such that $f = F|_{\mathbb{S}^n}$ and $f^2 = F_2|_{\mathbb{S}^n}$. So, $F^2|_{\mathbb{S}^n} = F_2|_{\mathbb{S}^n}$, polynomials $F^2$ and $F_2$ are homogeneous, hence $F^2 = F_2 r^{2k}$, where $r^2 = (x^1)^2 + ...+(x^{n+1})^2$. Suppose, that $k \neq 0$. The number $k$ is integer, because $F^2$ and $F_2$ are polynomials. If $n \geq 2$ then a polynomial $r^{2k}$ is irreducible. Then $F$ or $F_2$ is divisible by a polynomial $r^2$. This contradicts  \cite[Remark 2]{MR166304}, because $F$ and $F_2$ are harmonic. Hence, $k = 0$ and $F_2 = F^2$, so $F$ and $F^2$ are harmonic.  Finally, by Proposition \ref{propos} we have that $F$ is a $(0, 0)$-eigenfunction.

Let us prove the sufficiency.  By Proposition \ref{propos}  polynomial $F^2$ is harmonic, hence $f$ and $f^2$ are eigenfunctions for the operator $\tau$. Then by Proposition \ref{propos} a function $f$ is a $(\lambda, \mu)$-eigenfunction.

\end{proof}

\section{Application to  \texorpdfstring{$ \mathbb{S}^4 $}{S4} }

Theorem \ref{Sn_class} allows us to construct $(\lambda, \mu)$-eigenfunctions on $\mathbb{S}^n$. Then, minimal submanifolds of $\mathbb{S}^4$ could be obtained.

\begin{example}
    Consider a polynomials $F_k(x^1, x^2, x^3, x^4, x^5) = (x^1 + i x^2)^k + (x^3 + i x^4)^k$. By Theorems \ref{GM_th} and \ref{Sn_class} the submanifolds $F_k^{-1}(0) \cap \mathbb{S}^4 \setminus \{ (0, 0, 0, 0, \pm 1) \}$ are  minimal  in $\mathbb{S}^4$. But these submanifolds turn out to be quite trivial. Namely, these submanifolds  are unions of flat sections of $\mathbb{S}^4$ without two points.
\end{example}

\hypertarget{def2}{}
\begin{def1}[\cite{gudmundsson2023minimal}]
 Let $\varphi: (M, g) \rightarrow (N, h)$ be a smooth submersion between Riemannian manifolds.  Let the functions $k_1^2, ... , k_n^2$ denote the non-zero eigenvalues of the first fundamental form $\varphi^* h$ with respect to the metric $g$. Let $K$ be a submanifold of $M$. Then $\varphi$ is said to be \textit{horizontally conformal up to first order} along $K$ if for all $p \in K$ we have

\begin{enumerate}
    \item $k_1^2(p) = ... = k_n^2(p),$
    \item $\grad k_1^2(p) = ... = \grad k_n^2(p).$
\end{enumerate}
\end{def1}

It turns out that the following  proposition is true.

\begin{propositionf}
\label{propos_pol}
    
    Let $f$ be a $(\lambda, \mu)$-eigenfunction on $\mathbb{S}^4$ given by a polynomial $F$ of $\deg F \leq 3$, such that $0 \in f(\mathbb{S}^4)$ is a regular value.   Then $f^{-1}(0)$ is a  flat section of~$\mathbb{S}^4$.
\end{propositionf}

\begin{proof}
    Theorem \ref{Sn_class} implies that a polynomial $F$ is a $(0, 0)$-eigenfunction. According to \cite[Lemma 4.1]{gudmundsson2023minimal}, two eigenvalues $k^2_{1,2}$ of $F^*h$ satisfy
$$k_{1,2}^2 = \frac{1}{2}((|\grad u)|^2 + |\grad v|^2 \pm \mu (u^2 + v^2)).$$
Since $\mu = 0$, the polynomial $F$ is horizontally conformal up to first order everywhere in $\mathbb{R}^{n+1}$. Also the polynomial $F$ is homogeneous, then \cite[Theorem 4.3]{MR1704994} implies that there exists an affine change of coordinates in $\mathbb{R}^5$, such that in the new coordinates we have $\Bar{F} = \Bar{F}(\Bar{x}^1, \Bar{x}^2, \Bar{x}^3, \Bar{x}^4)$.   Consider $\mathbb{S}^3 = \mathbb{S}^4 \cap \{ \Bar{x}^5 = 0 \}$ and the $(\lambda, \mu)$-eigenfunction on it defined by $\Bar{F}$. So, $\Bar{F}^{-1}(0) \cap \mathbb{S}^3$ is a minimal submanifold of $\mathbb{S}^3$ of dimension 1. Hence, $\Bar{F}^{-1}(0) \cap \{ \Bar{x}^5 = 0 \}$ is a plane section of $\mathbb{R}^4 = \mathbb{R}^5 \cap \{ \Bar{x}^5 = 0 \}$. That means $\Bar{F}^{-1}(0)$ is a plane section of $\mathbb{R}^5$.
\end{proof}

Unfortunately there are no simple classification of horizontally
conformal polynomial maps of arbitrary degree. For more information see \cite[Section 5.2]{BW2003}.

\section{Minimal submanifolds of codimension 1 and complex-valued eigenfunctions}

In this Section we investigate using \cite[Theorem 3.3]{MR1190447} the question when preimage of a minimal submanifold of $\mathbb{C}$, i.e. a real straight line, is minimal. This allows us to construct minimal submanifolds of codimension 1.  

\begin{def1}[\cite{gudmundsson2023minimal}]
    Let $\varphi: (M, g) \rightarrow (N, h)$ be a smooth map between Riemannian manifolds. For a real number $p > 1$, we say that $\varphi$ is a \textit{p-harmonic map} if it is a critical point of the $p$-energy functional,
    $$E_p(\varphi) = \frac{1}{p} \int_M |d \varphi|^pdx.$$
\end{def1}

\begin{propositionf}[see e. g. \cite{gudmundsson2023minimal}]
    A smooth map $\varphi : (M, g) \rightarrow(N, h)$ between Riemannian manifolds is $p$-harmonic if and only if it satisfies the Euler-Lagrange equation for $E_p$,
    $$ \tau_p(\varphi) := |d\varphi|^{p-2}[\tau(\varphi) +d \varphi (\grad(log(|d\varphi|^{p-2})))] \equiv 0.$$
\end{propositionf}

\subsection{Proof of Theorem \ref{Sn_th}}

\begin{proof}
Firstly,  the polynomial $F$ is a $(0, 0)$-eigenfunction, hence $\tau(F) = 0$, so $F$ is 2-harmonic. Following the proof of Proposition \ref{propos_pol} we see, that $F$ is horizontally conformal up to first order everywhere in $\mathbb{R}^{n+1}$.  Secondly, using the formula
$$\grad^{\mathbb{R}^{n+1}} F(x)= \deg(F) \, F(x) \frac{x}{||x||^2}  + ||x||^{d-1} \grad^{\mathbb{S}^{n}} f \left ( \frac{x}{||x||}  \right )$$
we see that for all $x \in \mathbb{R}^{n+1} \setminus \{ 0 \} $ we have $\grad^{\mathbb{R}^{n+1}} F(x) \neq 0$, because $0$ is a regular value of $f$. Hence $F(x)$ is a submersion for   $x \in \mathbb{R}^{n+1} \setminus \{ 0 \} $. So, according to \cite[Theorem 3.3]{MR1190447} the submanifold $F^{-1}(l_\alpha) \setminus \{ 0 \}$ is minimal in $\mathbb{R}^{n+1}$ if and only if for all $x \in F^{-1}(l_\alpha) \setminus \{ 0 \}$ we have
$$    d  F_x( {\mathcal{H}_2} \grad  k^{-2}  ) = 0, $$
where $k^2 = k_1^2 = k_2^2$, $H_2 = (T_xF^{-1}(l_\alpha))^{\perp}$, so in our case $H_2 =  \Span(\grad P_{\alpha}(x))$ and $\mathcal{H}_2$ is the ortogonal projector on $H_2$. Now let us simplify this expression  using $k^2 =  \frac{1}{2}((|\grad u)|^2 + |\grad v|^2)$. Since $F$ is a $(0, 0)$-eigenfunction, we get $|\grad u| = |\grad v|$ and $\grad u$ is ortogonal to $\grad v$.  Then

$$|\grad P_{\alpha}|^2 = |a \grad u + b \grad v|^2 = (a^2 + b^2) |\grad u|^2 = (a^2 + b^2) k^2 = k^2.$$
 A straightforward calculation in standard coordinates in $\mathbb{R}^{n+1}$ at fixed point $x$ shows that 
$$\pm | {\mathcal{H}_2} \grad    k^{-2}   |=  
\left (\frac{\grad P_{\alpha}}{|\grad P_{\alpha}|} , \grad \left (  k^{-2}\right ) \right ) = $$
$$ = \left ( \frac{\grad P_{\alpha}}{|\grad P_{\alpha}|}, \grad(|\grad(P_{\alpha})|^{-2}) \right ) = \frac{1}{|\grad P_\alpha|} \delta^{ij} \frac{\partial P_\alpha}{\partial x^i}\frac{\partial }{\partial x^j} \left ( \delta^{kl} \frac{\partial P_\alpha}{\partial x^k} \frac{\partial P_\alpha}{\partial x^l} \right )^{-1} = $$
$$-\frac{1}{|\grad P_\alpha|^5} \delta^{ij} \delta^{kl} \left ( \frac{\partial^2 P_\alpha}{\partial x^j x^k} \frac{\partial P_\alpha}{\partial x^i} \frac{\partial P_\alpha}{\partial x^l} + \frac{\partial^2 P_\alpha}{\partial x^j x^l} \frac{\partial P_\alpha}{\partial x^i} \frac{\partial P_\alpha}{\partial x^k} \right ) =-2 \frac{ \Hess P_{\alpha} (\grad P_{\alpha}, \grad P_{\alpha})}{|\grad P_{\alpha}|^{5}}.$$

Also, remark that $ \mathcal{H}_2 \grad  k^{-2}$ is parallel to $\grad P_{\alpha}$. Hence, there exists a function $c(x)$ such that  $ {\mathcal{H}_2} \grad  k^{-2} = c(x) \grad P_{\alpha}$.  So, 
$$d  F_x( {\mathcal{H}_2} \grad k^{-2} ) =   c(x)(a |\grad u|^2, b |\grad v|^2).$$
Hence, $d  F_x({\mathcal{H}_2} \grad k^{-2} ) = 0$ if and only if $c(x) = 0$. So, $| {\mathcal{H}_2} \grad  k^{-2}   | = 0$ and $\Hess P_{\alpha} (\grad P_{\alpha}, \grad P_{\alpha}) = 0$. 

Finally, $F^{-1}(l_\alpha)$ is a cone over $F^{-1}(l_\alpha) \cap \mathbb{S}^n$, because $F$ is a homogeneous  polynomial. By Simons theorem \cite[Section 3.2]{MR1391729}, $F^{-1}(l_\alpha) \setminus B_{\varepsilon}(0)$ is minimal in $\mathbb{R}^{n+1}$ if and only if $F^{-1}(l_\alpha) \cap \mathbb{S}^n$ is minimal in $\mathbb{S}^n$. Here $B_\varepsilon(x)$ denotes the open ball of radius $\varepsilon$ with center $x$.

\end{proof}

\subsection{Proof of Theorem \ref{S3_th}}

\begin{proof}
    Let us prove the first statement. Since $F$ is a $(0, 0)$-eigenfunction, by \cite[Theorem 4]{riedler2023polynomial} there exists an orthogonal  change of coordinates in $\mathbb{R}^4$, such that in the new coordinates $F = F(\Tilde{z}^1, \Tilde{z}^2)$.  Since $f^{-1}(l_\alpha) = P^{-1}_\alpha(0) \cap \mathbb{S}^3$, it is enough to show that  submanifolds $P^{-1}_\alpha(0)$ and $ P^{-1}_\beta(0)$ are isometric. Let us consider $\mathbb{R}^4$ as $\mathbb{C}^2$, then the isometry  $G: \mathbb{C}^2 \rightarrow \mathbb{C}^2$ is given by
    $$G(z^1, z^2) := (\overline{\beta \alpha^{-1}})^{\frac{1}{\deg F}}(z^1, z^2).$$
A straightforward calculation shows that
\[ 
\begin{array}{l}
\vspace{2mm}
P_\alpha(G(z^1, z^2)) = a \Re(F(G(z^1, z^2))) + b \Im(F(G(z^1, z^2)))  \\
\vspace{2mm}
 \quad \quad \quad \quad \quad \quad = a \Re \left( \overline{\beta \alpha^{-1}} F \right ) + b \Im \left ( \overline{\beta \alpha^{-1}} F \right )  \\
 \vspace{2mm}
   \quad  \quad \quad \quad \quad \quad = \left [ a \Re \left ( \beta \alpha^{-1} \right ) -  b \Im \left ( \beta \alpha^{-1} \right ) \right] \Re F   \\
   \vspace{2mm}
   \quad \quad \quad \quad \quad \quad \quad + \left[ b \Re \left ( \beta \alpha^{-1} \right ) + a \Im \left ( \beta \alpha^{-1} \right ) \right] \Im F  \\

   \quad \quad \quad \quad \quad \quad = P_\beta(z^1, z^2).
  \end{array} 
\]
Hence, the map $G$ gives an isometry between  $P^{-1}_\alpha(0)$ and $ P^{-1}_\beta(0)$.

Now, let us prove the second statement. By Theorem \ref{Sn_th} submanifold $f^{-1}(l_{\alpha_0})$ is minimal in $\mathbb{S}^3$ if and only if for all $x \in P^{-1}_{\alpha_0}(0)$ we have $\Hess P_{\alpha_0}(\grad P_{\alpha_0} , \grad P_{\alpha_0})  = 0$.  But the submanifolds $f^{-1}(l_\alpha)$ are isometric, so submanifold $f^{-1}(l_{\alpha_0})$ is minimal in $\mathbb{S}^3$ for some $\alpha_0$ if and only if $f^{-1}(l_\alpha)$ is minimal in $\mathbb{S}^3$ for every $\alpha$. 
\end{proof}

\subsection{Examples in $\mathbb{S}^3$}

$ $  \par Now let us construct several examples of minimal submanifolds in $\mathbb{S}^3$ with the help of  $(\lambda, \mu)$-eigenfunctions. Let $z^1 = x^1 + i x^2$ and $z^2 = x^3 + i x^4$.

\begin{remark}
\label{r_cr}
    In some cases $0 $ is not a regular value of an $(\lambda, \mu)$-eigenfunction $f$. However it is easy to check that Theorem \ref{Sn_th} and Theorem \ref{S3_th} hold for $f^{-1}(l_\alpha) \setminus~ Cr(f)$, where $Cr(f)$ is a set of critical points of a function $f$.
\end{remark}

\begin{example}[The Clifford  torus]
     Consider a  $(\lambda, \mu)$-eigenfunction $f$ on $\mathbb{S}^3$ given by a polynomial $F = (z^1)^2 + (z^2)^2$. It is easy to compute that $\Hess P_1 (\grad P_1, \grad P_1) = 8 P_1$. Hence, by Theorem \ref{S3_th} the submanifold $F^{-1}(l_\alpha) \cap \mathbb{S}^3$ is minimal in $\mathbb{S}^3$ for every~$\alpha$. Consider $\alpha = 1$ and $\alpha = i$,
$$P_1 = (x^1)^2 - (x^2)^2 + (x^3)^2 - (x^4)^2, \quad P_i = 2 x^1 x^2 + 2 x^3 x^4 .$$
So, $P_1^{-1}(0) \cap \mathbb{S}^3$ and $P_i^{-1}(0) \cap \mathbb{S}^3$ are minimal submanifolds of $\mathbb{S}^3$. It is easy to check that  $P_1^{-1}(0) \cap \mathbb{S}^3$ corresponds to  the Clifford  torus immersed by first  eigenfunctions and $P_i^{-1}(0) \cap \mathbb{S}^3$ corresponds to  the Clifford  torus immersed by second  eigenfunctions. 

Also remark that  the first statement of Theorem \ref{S3_th} implies that there exists an isometry of $\mathbb{S}^3$ that induces an isometry between these tori. In this case the isometry is given by $G(z^1, z^2) = \frac{1 - i}{\sqrt{2}}(z^1, z^2)$.
\end{example}

\begin{example}[Lawson $\tau_{n, m}$ surfaces]
    
Consider a  $(\lambda, \mu)$-eigenfunction $f$ on $\mathbb{S}^3$ given by a polynomial $F = (z^1)^n (z^2)^m$. This example is based on \cite[Example 3.5]{MR1190447}. One can apply the Simons theorem \cite[Section 3.2]{MR1391729} directly to it and see that $f^{-1}(l_\alpha)$ is minimal in $\mathbb{S}^3$ for every $\alpha$, but we show how to do it using Theorem \ref{S3_th}. Consider the  coordinates $(r^1, r^2, \varphi^1, \varphi^2)$ in $\mathbb{R}^4$ defined by
    $$(x^1, x^2, x^3, x^4) = (r^1 cos\varphi^1, \, r^1 sin\varphi^1, \, r^2 cos\varphi^2, r^2 sin\varphi^2).$$
Write $P_1 = \Re F$ in these coordinates,
$$P_1 = \Re[(r^1)^n (r^2)^m(cos(n \varphi^1) + i sin(n \varphi^1)) (cos(m \varphi^2) + i sin(m \varphi^2))] =$$
$$= (r^1)^n (r^2)^m cos(n \varphi^1 + m \varphi^2) =(r^1)^n (r^2)^m cos\varphi, $$
where $\varphi = \varphi(\varphi^1, \varphi^2) := n \varphi^1 + m \varphi^2$. Then calculations show that
$$\grad P_1 = (r^1)^{n} (r^2)^{m} \left (\frac{ n \, cos(\varphi)}{r^1}, \frac{m \,cos(\varphi)}{r^2}, -\frac{ n \, sin(\varphi)}{(r^1)^2}, -\frac{m \,sin(\varphi)}{ (r^2)^2} \right ),$$
\begin{multline*}
\Hess P_1(\grad P_1, \grad P_1) = \\ =(r^1)^{3n -4} (r^2)^{3m - 4}  cos\varphi \, \Big [ m^3(m-1)(r^1)^4  + 2 m^2 n^2 (r^1r^2)^2 + n^3(n - 1)(r^2)^4   \Big ].
\end{multline*}

So for all $x \in P_1^{-1}(0) $ we have $\Hess P_1(\grad P_1, \grad P_1) = 0$. By Theorem \ref{S3_th} and Remark \ref{r_cr} the submanifold $f^{-1}(l_\alpha) \setminus (\{ z^1 = 0 \} \cup \{ z^2 = 0\} )$ is minimal in $\mathbb{S}^3$ for all~$\alpha$. This family contains all  Lawson $\tau_{n, m}$ surfaces. In order to  get the canonical equation for a Lawson surface, suppose that $\alpha = i$. So, after the reflection with respect to the plane $Ox^1x^2x^3$, we get $\Im((z^1)^n (\Bar{z^2})^m) =0$. Hence, by the Lawson classification \cite[Theorem 3]{MR270280} we have the following list.

  \begin{enumerate}
        \item If $m = 0 $ or $n = 0$, this is $\mathbb{S}^2$.
        \item If $m, \, n \in \mathbb{N}$ and $2 \nmid m n$, this is  $\mathbb{T}^2$.
        \item  If $m, \, n \in \mathbb{N}$ and $2 \, |  \, m n$, it is  ${\mathbb{K}}^2$, i.e. the Klein bottle.
    \end{enumerate}

\end{example}

\section*{Acknowledgements}
The work was supported by the Theoretical Physics and Mathematics
Advancement Foundation «BASIS» grant Leader (Math) 21-7-1-45-6 and Stipend (Student) 24-8-2-19-1.

The author would like to thank A. V. Penskoi for attracting attention to this problem and valuable discussions.

\bibliographystyle{alpha}
\bibliography{bibl.bib}

\end{document}